\documentclass[letterpaper, 10 pt, conference]{ieeeconf}  

\IEEEoverridecommandlockouts                              
                                                          
\usepackage{etoolbox}    
\newtoggle{report}
\newtoggle{extensions}
\toggletrue{report}

\newcommand{\reportOnly}[1]{
    \iftoggle{report}{#1}{}
}
\newcommand{\submissionOnly}[1]{
    \iftoggle{report}{}{#1}
}

\usepackage[english]{babel}
\usepackage[utf8]{inputenc}

\usepackage{amsthm,amsmath,amssymb}
\makeatletter
\renewcommand\th@plain{\normalfont}
\makeatother

\usepackage{hyperref}
\hypersetup{
    colorlinks=true,
    linkcolor=blue,
    filecolor=magenta,      
    urlcolor=cyan,
    citecolor=orange
    }
\usepackage{csquotes}
\usepackage{xfrac}
\usepackage{fancyhdr}
\usepackage{enumerate}
\let\labelindent\relax
\usepackage[shortlabels]{enumitem}
\usepackage{multirow}
\usepackage{stmaryrd} 
\usepackage{cancel}
\usepackage[dvipsnames]{xcolor}
\usepackage{empheq}
\usepackage{hyperref}
\usepackage{caption}
\usepackage{subcaption}
\usepackage{graphics}
\usepackage{mathrsfs}

\usepackage[url = false,%
            doi = false,%
            isbn = false,%
            natbib,%
            style=ieee%
            ]{biblatex}
\setlist{nosep,leftmargin=*}
\usepackage[font=small,labelfont=bf,justification=justified]{caption}

\usepackage[ 
    final,
    commandnameprefix=ifneeded, 
    xcolor=dvipdf,
    todonotes={colorinlistoftodos,
    prependcaption,
    textsize=footnotesize,
    backgroundcolor=orange!10,
    textcolor=black,
    linecolor=orange,
    bordercolor=orange}
]{changes}

\allowdisplaybreaks{}

\def\BibTeX{{\rm B\kern-.05em{\sc i\kern-.025em b}\kern-.08em
T\kern-.1667em\lower.7ex\hbox{E}\kern-.125emX}}

\newcounter{thmcounter}[section]
\renewcommand{\thethmcounter}{\arabic{section}.\arabic{thmcounter}}

\newtheorem{theorem}[thmcounter]{Theorem}

\newtheorem{lemma}[thmcounter]{Lemma}
\newtheorem{corollary}[thmcounter]{Corollary}
\newtheorem{definition}[thmcounter]{Definition}
\newtheorem{assumption}[thmcounter]{Assumption}

\newtheorem{remark}[thmcounter]{Remark}
\newtheorem{example}[thmcounter]{Example}

\newcommand{\R}{\mathbb{R}}
\newcommand{\N}{\mathbb{N}}

\newcommand{\Q}{\mathbb{Q}}

\renewcommand{\cal}[1]{\mathcal{#1}}

\newcommand{\setto}{\rightrightarrows}
\renewcommand{\emptyset}{\varnothing}

\DeclareMathOperator*\dom{\operatorname{dom}}

\DeclareMathOperator*{\rot}{rot}

\newcommand{\inner}[2]{\left\langle #1, #2 \right\rangle}

\setkeys{Gin}{draft=false}  
\graphicspath{ {./Figs/} }  

\begin{document}

\title{On the Instability of Nesterov's ODE under Non-Conservative Vector Fields}
\author{Daniel E. Ochoa, Mahmoud Abdelgalil, Jorge I. Poveda
    \thanks{\scriptsize D. E. Ochoa is with the Department of Electrical
        and Computer Engineering, University of California Santa Cruz, 95064, USA. M. Abdelgalil and J. I. Poveda are with the Department of Electrical and Computer Engineering, University of California San Diego, La Jolla, CA 92093. This work was supported in part by NSF grants ECCS CAREER 2305756, CMMI 2228791, and AFOSR grant FA9550-22-1-0211.}
}
\maketitle

\thispagestyle{empty} 
\begin{abstract}
    We study the \emph{instability} of Nesterov's ODE in non-conservative settings, where the driving term is not necessarily the gradient of a potential function. While convergence properties under Nesterov's ODE are well-characterized for settings with gradient-based driving terms, we show that the presence of \emph{arbitrarily small} non-conservative terms can lead to instability. To resolve the instability issue, we study a regularization mechanism based on restarting. For this mechanism, we establish novel explicit bounds on the resetting period that ensure the decrease of a suitable Lyapunov function, thereby guaranteeing stability and ``accelerated'' convergence rates under suitable smoothness and monotonicity properties on the driving term. Numerical simulations support our results.
\end{abstract}
\section{Introduction}
The Nesterov Accelerated Gradient method has been a cornerstone of optimization since its introduction in \cite{nesterovMethodSolvingConvex1983}. Its continuous-time analog, introduced in \cite{suDifferentialEquationModeling2016} and termed the \emph{Nesterov's Ordinary Differential Equation} (ODE), has emerged as a powerful tool to study optimization algorithms using a continuous-time dynamical systems point of view \cite{shiUnderstandingAccelerationPhenomenon2022}.
Nesterov's ODE is defined by the equation
\begin{equation}\label{eq:initial_Nesterov_ODE0}
    \ddot{x}+ \frac{3}{\tau}\dot{x} + \cal G(x) = 0,\quad \added{\dot\tau = \eta},
\end{equation}
where $x \in \mathbb{R}^n$, \added{$\tau\in[T_0,\infty)$}, with $T_0>0$, $\cal G: \mathbb{R}^n \to \mathbb{R}^n$, \added{and $\eta\in (0,1]$}. When $\cal G(x) = \nabla J(x)$ for some \added{convex} potential function $J$, Nesterov's ODE achieves accelerated convergence to the minimum of $J$ with rate $O(1/t^2)$ \cite{suDifferentialEquationModeling2016}.

The remarkable success of Nesterov’s method in optimization motivates the exploration of its potential extensions to \emph{non-conservative settings}, where the vector field~$\mathcal{G}$ is not necessarily the gradient of a potential function. Such scenarios naturally arise in areas such as game theory and consensus-based distributed optimization~\cite{gharesifardDistributedContinuousTimeConvex2014}. In game-theoretic contexts, it is well known that for a class of games known as \emph{potential games}, convergence to equilibria can be achieved via gradient-like dynamics. Similarly, in multi-agent systems with undirected communication graphs, consensus dynamics are often interpreted as gradient flows. These observations raise an important question: can dynamical systems of the form~\eqref{eq:initial_Nesterov_ODE0} also provide benefits in settings where the game is not potential or the communication graph is directed?

Extending Nesterov’s acceleration to such domains, however, presents several challenges. First, the lack of a potential function prevents one from formulating the ODE within an optimization framework and leveraging the body of existing stability results developed in that context. Second, while in optimization problems Nesterov’s ODE achieves accelerated convergence through dynamic damping of the form~$\frac{3}{t}\dot{x}$, this same term can have adverse effects in non-conservative systems. For instance, our previous studies on accelerated Nash equilibrium seeking~\cite{ochoaMomentumBasedNashSetSeeking2024} and distributed concurrent learning~\cite{ochoaDecentralizedConcurrentLearning2024} revealed that this damping mechanism fails to mitigate the destabilizing influence of the non-conservative component of~$\mathcal{G}$ in~\eqref{eq:initial_Nesterov_ODE0}, often leading to instability. Although similar behaviors have been observed numerically even in optimization problems~\cite{poveda2019inducing}, a theoretical explanation for this \emph{instability phenomenon} has remained elusive.

The first contribution of this paper is to provide a theoretical explanation for the emergence of instability in Nesterov’s ODE when the term~$\mathcal{G}$ is non-conservative. Specifically, by leveraging the Helmholtz decomposition theorem~\cite{glotzlHelmholtzDecompositionPotential2023}, we decompose~$\mathcal{G}$ into its conservative and non-conservative components, and employ the variation-of-constants formula~\cite[Prop.~9.6]{bulloGeometricControlMechanical2005} to obtain a representation suitable for averaging analysis. We then study the average system using Floquet theory~\cite[Sec.~19]{sarychevLieChronologicoalgebraic2001} to identify subclasses of vector fields~$\mathcal{G}$ for which Nesterov’s ODE \emph{fails} to stabilize the set~$\mathcal{A}\coloneqq\{x:\mathcal{G}(x)=0\}$. \added{We further show that, in specific scenarios, this instability persists even when the non-conservative component of~$\mathcal{G}$ is arbitrarily small compared to the conservative component, thus providing new insights into adversarial perturbations that can destabilize such accelerated flows in traditional optimization settings. Our instability analysis focuses on linear non-conservative vector fields whose conservative components are strongly monotone.}

The second main contribution is the development and analysis of a hybrid dynamical system that stabilizes~$\mathcal{A}$ by combining two mechanisms: continuous-time flows governed by Nesterov-like dynamics, and discrete-time resets triggered whenever $\tau\in\mathbb{R}_{\ge0}$ reaches the upper bound of a compact interval. \added{Unlike the instability analysis, which focuses on linear vector fields, this hybrid framework encompasses a broader class of nonlinear vector fields satisfying suitable monotonicity and Lipschitz conditions.} Extending the results in~\cite{ochoaMomentumBasedNashSetSeeking2024}, we derive quasi-optimal reset conditions and establish improved convergence rate guarantees.

The paper is organized as follows. Section~\ref{sec:preliminaries} presents the preliminaries. Section \ref{sec:standardAveraging} prepares Nesterov's ODE for averaging, and Section \ref{sec:instability} presents its instability properties. Section \ref{sec:stabilityRestarting} presents a hybrid mechanism to recover stability. Section \ref{sec:numeric} presents a numerical example. Section \ref{sec:conclusions} ends with conclusions.
\section{Preliminaries}
\label{sec:preliminaries}
We use $(u,v)=[u^{\top},v^{\top}]^{\top}$to denote the concatenation of $u,v\in\R^n$. \added{The Euclidean inner product is denoted by $\inner{u}{v}$ for all $u,v\in \R^n$}. We let $|z|\coloneqq \sqrt{\inner{z}{z}}$ for each $z \in \mathbb{R}^n$, and use $|z|_C \coloneqq \min_{s \in C} |z - s|$ to refer to the \added{minimum} distance between $z$ and a closed set $C$. \added{Given a matrix $A \in \mathbb{R}^{m \times n}$, the induced 2-norm is defined by $\|A\| \coloneqq \sup_{|x|=1} |Ax|$.} The set of eigenvalues of a matrix $A\in\R^{n\times n}$, counted with multiplicities, is denoted by $\operatorname{spec}(A)$. We use $A_{ij}\in \mathbb{R}$ for the $i$-th entry of the $j$-th column of $A\in\R^{n\times n}$. Given $\{a_k\}_{k=1}^n\subset \R^n$,  $\operatorname{diag}\{a_k\}_{k=1}^n\in \R^{n\times n}$ denotes the matrix whose $i$-th diagonal entry equals $a_i$ for all $i\in\{1,2,\ldots, n\}$, and whose off-diagonal entries are zero. A set-valued map $F:\R^n\setto\R^n$ assigns to each point $x\in \R^n$ a set $F(x)\subset \R^n$. 
Given $f:\mathcal{X}\to \mathcal{Y}$ and $g:\mathcal{Y}\to \mathcal{Z}$, we use $g\circ f$ to denote their composition. The \emph{flow $\Phi_s^f$ along a vector field} $f:\mathbb{R}^n \to \mathbb{R}^n$ assigns to each point $(s, x_0)$ the value $\Phi_{s}^f(x_0) = \chi(s) \in \mathbb{R}^n$, where $\chi$ is the unique solution to $\dot{x} = f(x)$ satisfying $\chi(0) = x_0$. Given a diffeomorphism $\Psi: \mathbb{R}^n \to \mathbb{R}^n$ and $f: \mathbb{R}^n \to \mathbb{R}^n$, the \emph{pullback} of $f$ by $\Psi$, denoted $\Psi^*f: \mathbb{R}^n \to \mathbb{R}^n$, is defined by $(\Psi^*f)(x) = \left(\frac{\partial \Psi^{-1}}{\partial x}\circ f \circ \Psi\right)(x)$ for all $x\in \mathbb{R}^n$, where $\frac{\partial \Psi^{-1}}{\partial x}(y)$ denotes the Jacobian matrix of $\Psi^{-1}$ evaluated at $y \in \mathbb{R}^n$. \added{Given} a smooth function $J:\mathbb{R}^n\to \mathbb{R}$ we define its gradient by \added{$\nabla J(x) = \left(\frac{\partial J}{\partial x_1}(x), \ldots, \frac{\partial J}{\partial x_n}(x)\right)$}. \added{Given a function $f:\R^n\to \R^n$, we define its divergence} $\operatorname{div} f(x)\coloneqq \sum_{i=1}^n\frac{\partial f_i}{\partial x_i}(x)$, and let the rotation operator $\overline{\rot }\, f:\R^n\to\R^{n\times n}$ be given by $(\overline{\rot }\, f(x))_{ij} \coloneqq \frac{\partial f_i(x)}{\partial x_j}-\frac{\partial f_j(x)}{\partial x_i}$ for each $x\in\R^n$. Given a matrix-valued function $K:\mathbb{R}^n\to \mathbb{R}^{n\times n}$, we define the rotation operator as $\rot K(x) = \sum_{j=1}^n \left(\frac{\partial K_{1j}}{\partial x_1}(x),\ldots, \frac{\partial K_{nj}}{\partial x_n}(x) \right)$, where $K_{ij}:\mathbb{R}^n\to \mathbb{R}$ corresponds to the function $x\mapsto (K(x))_{ij}$.
%
%

To obtain our results, we use the following theorem relating the flows of \added{two vector fields $g,f:\R^n\to\R^n$}:
%
\begin{theorem}[{Variation-of-Constants Formula~\cite[Prop. 9.6]{bulloGeometricControlMechanical2005}}]\label{variationOfConstants}
    Let $g$ be $r$-th and $f$ be $(r+1)$-th continuously differentiable time-dependent vector fields. Let $x_0 \in \mathbb{R}^n$, $T > 0$, $\tau \in [0,T]$, and define $\tilde{g}(\tau,x) = ((\Phi^f_{\tau})^*g_\tau)(x)$, where $g_\tau(x)\coloneqq g(x,\tau)$, and $(\Phi^f_{\tau})^* g_\tau$ denotes the pullback of $g_\tau$ along the flow of $f$. Then, for all $\tau \in [0, T]$ where the flows $\Phi^{f+g}_\tau$, $\Phi_\tau^f$, and $\Phi^{\tilde{g}}_\tau$ exist, it follows that
    $
        \Phi^{f+g}_{\tau}(x_0) = \left(\Phi^f_{\tau} \circ \Phi^{\tilde{g}}_{\tau}\right)(x_0).
    $\hfill$\square$
    %
\end{theorem}

\emph{Hybrid Dynamical Systems:} To formulate mechanisms that recover the stability of Nesterov's ODE, we use hybrid dynamical systems (HDS) with state $x\in\mathbb{R}^n$ and dynamics
\begin{equation}\label{prelim:HDS0}
    \cal H:\begin{cases}
         & x \in C\qquad ~~\dot{x} \in F(x) \\
         & x \in D\qquad x^+ \in G(x)
    \end{cases},
\end{equation}
where $F:\mathbb{R}^n\rightrightarrows\mathbb{R}^n$ is the flow map, $G:\mathbb{R}^n\rightrightarrows\mathbb{R}^n$ is the jump map, $C\subset\mathbb{R}^n$ is the flow set, and $D\subset\mathbb{R}^n$ is the jump set. We use the tuple $\cal H=(C,F,D,G)$ to refer to the data of the HDS. Solutions to HDS are parameterized by continuous-time $t\in \mathbb{R}_{\ge 0}$ and discrete-time $j\in \mathbb{Z}_{\ge0}$ and thus evolve on hybrid time domains. For a formal definition of solutions to HDS we refer the reader to \cite[Sec. 2]{goebelHybridDynamical2012}. \added{Standard continuous-time systems described by the dynamics $\dot x=f(x)$ with $f:\mathbb{R}^n\to \mathbb{R}^n$, can be cast as HDS by setting $C=\mathbb{R}^n$, $D=\emptyset$, and $F(x)=\{f(x)\}$}.
\begin{definition}[Stability/Instability Notions]
    \added{A compact set $\cal A \subset \R^n$ is said to be uniformly globally exponentially stable (UGES) for the system $\cal H$ in \eqref{prelim:HDS0} if there exist $c_1,c_2>0$ such that
        $|x(t,j)|_\cal A \leq c_1|x(0,0)|_\cal Ae^{-c_2(t + j)}$ for all $(t,j)\in \dom x$
        and each solution $x$ to $\cal H$.} \added{The set $\cal A$ is unstable for $\cal H$ if there exists $\varepsilon \!>\! 0$ such that for each $\delta \!>\! 0$ there exist a solution $x$, with $|x(0,0)|_\cal A \!<\! \delta$, and $(T,J) \in\operatorname{dom} x$ satisfying $|x(T,J)|_\cal A \!>\! \varepsilon.$}  \hfill$\square$%
\end{definition}
\section{Nesterov's ODE:\\Analysis via Standard Averaging}
\label{sec:standardAveraging}
In this section, we prepare Nesterov's ODE for analysis via averaging. First, we transform \eqref{eq:initial_Nesterov_ODE0} to the $\tau$-timescale by noting that every \added{solution $(x,\tau)$ to \eqref{eq:initial_Nesterov_ODE0} satisfies} $\tau(t)=\added{\eta t + T_0}$, with \added{$T_0>0$}, for all $t\in \dom (x,\tau)$. This \added{timescale change generates the following} system with state  $(x,\tau)\in \R^n\times [T_0,\infty)$ and dynamics:
\begin{equation}\label{eq:initial_Nesterov_ODE}
    \frac{d^2x}{d\tau^2} + \frac{3}{\tau}\frac{dx}{d\tau} + \gamma\cal G(x) = 0,\quad \added{\frac{d\tau}{d\tau}=1},
\end{equation}
where $\gamma\coloneqq \eta^2$. To analyze the effect of the vector field $\cal G$ on the dynamics \eqref{eq:initial_Nesterov_ODE}, we use the following generalization of the Helmholtz decomposition theorem \cite[Thm. 7.2]{glotzlHelmholtzDecompositionPotential2023}:
\begin{lemma}[Helmholtz Decomposition]\label{lemma:helmholtzDecomposition}
    Suppose $\cal G$ is analytic \added{in $\R^n$}. Then, there exist functions \added{$P\!:\!\mathbb{R}^n\to \mathbb{R}^{n\times n}$, $J\!:\!\mathbb{R}^n\to \mathbb{R}$, and $K\!:\!\mathbb{R}^n\to \mathbb{R}^{n\times n}$ such that $\cal G(x) = \nabla J(x) + \rot  K(x)$ for all $x\in\mathbb{R}^n$, where $J(x) \!\coloneqq\! \sum_{i=1}^n P_{ii}(x)$ and $K(x)\! \coloneqq P(x) \!-\! P(x)^\top$ satisfy $\operatorname{div}(\rot K(x))\!=\! 0$ and $\overline{\rot }(\nabla J) \!=\! 0$ for all $x\in \R^n$.} \hfill$\square$
\end{lemma}
%
\added{To ensure uniqueness of solutions to \eqref{eq:initial_Nesterov_ODE} and guarantee that the flow along the corresponding vector field is well-defined, we impose the following conditions on the elements of the decomposition in Lemma~\ref{lemma:helmholtzDecomposition}.}
\begin{assumption}\label{assump:regularity_Helmholtz}
    \begin{enumerate}[(i)]
        \item \emph{Strong Monotonicity:} $\exists \kappa_{J}>0$ such that
              $
                  \left\langle \nabla J(x_1) - \nabla J(x_2), x_1 - x_2\right\rangle \ge \kappa_{J}|x_1-x_2|^2
              $
              for all $x_1,x_2\in\mathbb{R}^n$. Additionally, we have that
              $
                  \left\langle \rot  K(x_1) - \rot  K(x_2), x_1-x_2 \right\rangle\ge 0
              $
              for all $x_1,x_2\,{\in}\,\mathbb{R}^n$.
        \item \emph{Lipschitz continuity:} $\exists \ell_{J},\ell_{R}>0$ such that
              $
                  |\nabla J(x_1) - \nabla J(x_2)| \le \ell_{J}|x_1 - x_2|$ and $|\rot K(x_1) - \rot K(x_2)| \le \ell_{K}|x_1 - x_2|~\forall x_1,x_2\in\mathbb{R}^n.
              $
        \item \emph{Scaling relationship:} $\ell_{K} = \alpha\sqrt{\ell_{J}}$ with $\alpha\in (0,1]$.\hfill $\square$
    \end{enumerate}
\end{assumption}
Assumption \ref{assump:regularity_Helmholtz}-(i) ensures that the scalar function $J$ is strongly convex, which, via \cite[Thm. 3]{suDifferentialEquationModeling2016}, \added{guarantees accelerated convergence of Nesterov's ODE to the unique minimizer of $J$ when the non-conservative term $\rot K$ is not present.} The next example presents a class of vector fields that satisfies these conditions.
\begin{example}\label{example:linearDecomposition}
    Given $Q\in\mathbb{R}^{n\times n}$, the vector field $\cal G(x) = Qx$ is analytic and satisfies Assumption \ref{assump:regularity_Helmholtz}. The Helmholtz decomposition in Lemma \ref{lemma:helmholtzDecomposition} can be explicitly constructed by letting
    $J(x) = x^\top Q_s x$, $(K(x))_{ij} = \added{\frac{1}{2}\left(Q_a\right)_{ij}|x|^2} $ for all  $x\in \mathbb{R}^n \text{ and } i,j\in \{1,2,\ldots, n\}$,
    where $Q_s \coloneqq \frac{1}{2}(Q+Q^\top)$ and $Q_a \coloneqq \frac{1}{2}(Q-Q^\top)$.\hfill$\square$
\end{example}
%
Using Lemma \ref{lemma:helmholtzDecomposition} and Assumption \ref{assump:regularity_Helmholtz}, by letting $\added{s} \coloneqq \added{\tau}/\varepsilon$, with $\varepsilon \coloneqq (\ell_{J})^{-\frac{1}{2}}$, and letting $y \coloneqq (y_1, y_2) = (x, \frac{dx}{ds})\in \mathbb{R}^{2n}$, from \eqref{eq:initial_Nesterov_ODE} we obtain that \added{$d\tau/ds =\varepsilon$} and
\begin{align}\label{eq:pre_standard_Averaging_Form_Nesterov}
    \frac{dy}{d\added{s}} & =
    \begin{pmatrix}
        y_2 \\
        - \nabla \hat{J}(y_1)
    \end{pmatrix} + \varepsilon \begin{pmatrix}
                                    0 \\
                                    -\frac{3}{\varepsilon\added{s} + \added{T_0}}y_2 - \rot \hat{K}(y_1)
                                \end{pmatrix}\notag \\
                          & \eqqcolon h(y) + \varepsilon u(y,\varepsilon s),
\end{align}
for all $y\!\in\!\R^n$, where $\hat{J}(x)\!\coloneqq\!\frac{\added{\gamma}}{\ell_{J}}J(x)$ and $\hat{K}(x)\!\coloneqq\!\frac{\alpha\added{\gamma}}{\ell_{K}}K(x).$\hfill\break
%
By using the pullback of $u_{\varepsilon\added{s}}(y) \coloneqq u(y,\varepsilon\added{s})$ via the flow along $h$, we obtain a system with state $(z,\tau)\in \mathbb{R}^{2n}\times [T_0,\infty)$ and dynamics given by
            \begin{equation}\label{eq:standard_Averaging_Form_Nesterov}
                \frac{dz}{d\added{s}} = \varepsilon \left(\left(\Phi_\tau^{h}\right)^* u_{\varepsilon\added{s}}\right)(z) \eqqcolon \varepsilon v(z,\varepsilon\added{s},\added{s}),\quad \added{\frac{d\tau}{ds}=\varepsilon}.
            \end{equation}
            By Theorem \ref{variationOfConstants}, a solution $y$ to system \eqref{eq:pre_standard_Averaging_Form_Nesterov} and a solution $z$ to system \eqref{eq:standard_Averaging_Form_Nesterov} with $y(0)=z(0)$, satisfy $y(\added{s}) = \Phi_{\added{s}}^{h}\left(z(\added{s})\right)$ for all $\added{s}\in \dom y$.
            \begin{remark}
                The \added{system} in \eqref{eq:standard_Averaging_Form_Nesterov} admits a series representation that expands the pullback in integrals of iterated Lie brackets between $u_{\varepsilon\added{s}}$ and $h$.\reportOnly{\cite{agracevExponentialRepresentationFlows1979}}\hspace{-4pt}While closed forms are generally not available, simplifications emerge when certain Lie brackets vanish; an approach used for Euler-Lagrange systems in \cite{bulloAveragingVibrationalControl2002}. Although Nesterov's ODE has an Euler-Lagrange formulation that might benefit from similar simplifications, \added{our instability analysis studies the linear case of Example \ref{example:linearDecomposition}, leaving the nonlinear case for future work.}$\hfill\square$
            \end{remark}
            \section{Instability under Linear Mappings}\label{sec:instability}
            In this section, we study the setting where $\cal G(x)=Qx$ for some $Q\in \mathbb{R}^{n\times n}$ \added{and all $x\in \R^n$}. In this case, the Helmholtz decomposition yields $J(x) = \frac{1}{2}x^\top Q_sx$ and $\rot K(x) = Q_ax$, where $Q_s$ and $Q_a$ are the symmetric and skew-symmetric parts of $Q$, respectively. \added{Therefore}, the ODE in \eqref{eq:pre_standard_Averaging_Form_Nesterov} becomes
            \begin{align}\label{eq:preStandardLinear}
                \frac{dy}{d\added{s}} = \underbrace{\begin{pmatrix}
                                                            0          & I \\
                                                            -\hat{Q}_s & 0
                                                        \end{pmatrix}}_{\displaystyle\eqqcolon A}y + \varepsilon\underbrace{\begin{pmatrix}
                                                                                                                                0          & 0                                              \\
                                                                                                                                -\hat{Q}_a & -\frac{3}{\varepsilon\added{s} + \added{T_0}}I
                                                                                                                            \end{pmatrix}}_{\displaystyle\eqqcolon B(\varepsilon s)}y,
            \end{align}
            with $d\tau/ds=\varepsilon$, and where $\hat{Q}_s=\frac{Q_s}{\|Q_s\|}$, $\hat{Q}_a=\alpha\frac{Q_a}{\|Q_a\|}$, with $\alpha\in(0,1]$. Similarly, the dynamical system in \eqref{eq:standard_Averaging_Form_Nesterov} reduces to
\begin{align}\label{eq:standardLinear}
    \frac{dz}{d\added{s}} = \varepsilon \Lambda(\tau, s)z\coloneqq \varepsilon \left(e^{-As }B(\tau)e^{As}\right)z,\quad \frac{d\tau}{d\added{s}} = \varepsilon.
\end{align}
Before presenting the first main result of this paper, we introduce a lemma that, under suitable conditions on $Q_s$, shows that the flow of the vector field $y\mapsto Ay$ is periodic.
\begin{lemma}\label{lemma:periodic}
    Let $\cal G(x) = Qx$ where $Q\in \mathbb{R}^{n\times n}$ and suppose $\cal G$ satisfies Assumption \ref{assump:regularity_Helmholtz}. {\color{black}Suppose there exists $\mu\in \mathbb{R},\mu\neq 0,$ such that for each $\lambda \in \mathrm{spec}(Q_s)$ there exists $c_\lambda\in \Q_{>0}$\footnote{$\Q_{>0}$ denotes the set of positive rational numbers} satisfying $\lambda = (c_\lambda\mu)^2$}. Then, \added{there exists $\cal T>0$ such that} every solution $\psi$ to
    \begin{align}\label{eq:driftDynamics}
        \frac{d \psi}{d\added{s}} = A\psi,
    \end{align}
    where $A$ is defined in \eqref{eq:preStandardLinear}, is periodic \added{with period $\cal T$}.\hfill$\square$
\end{lemma}
\added{
    \begin{proof}
        By the definition of $A$ and the determinant formula for block matrices,
        $
            \text{det}\left(\lambda I - A\right)  = \text{det}(\lambda^2 I + \hat{Q}_s).
        $
        Thus, the eigenvalues of $A$ are $\pm i\omega_j$, where $\omega_j^2 = \lambda_j$ for $\lambda_j \in \mathrm{spec}(\hat{Q}_s)$.
        Since $J(x) = \frac{1}{2}x^\top Q_s x$ is strongly convex, $\hat{Q}_s$ is positive definite, so all $\omega_j \in\R_{>0}$. Since $\hat{Q}_s$ is positive definite, there exists an orthogonal matrix $U$ that diagonalizes $\hat{Q}_s$. In the coordinates defined by $U$, the system \eqref{eq:driftDynamics} decouples into planar rotation blocks of the form $\dot \psi_j = \begin{pmatrix}0 & -\omega_j\\ \omega_j & 0\end{pmatrix} \psi_j$. Thus, each solution $\psi$ to \eqref{eq:driftDynamics} is a superposition of periodic signals with frequencies $\{\omega_j\}_{j=1}^n$.
        By the assumption on $Q_s$, we have that $\lambda_j = (c_j\mu)^2$ for $c_j\in\Q, c_j > 0$. Then, writing $c_j = a_j/b_j$ with $a_j, b_j \in \N_{>0}$, for each $j\in\{1,2,\ldots, n\}$ there exists $k_j \in \mathbb{Z}$ such that $\omega_j = k_j\omega_0$ where $\omega_0 \coloneqq \mu/L$, with $L$ being the least common multiple of $\{b_1,b_2,\ldots,b_n\}$. This result ensures periodicity of each solution $\psi$, with period $\cal T = 2\pi/\omega_0$.
    \end{proof}}
The result of Lemma \ref{lemma:periodic} enables the use of techniques for averaging of periodic systems with slow-time dependence (see \cite[Sec. 3.3]{sandersAveragingMethodsNonlinear2007}), and leads to our first main result.
\begin{theorem}\label{maintheorem}
    Let $\cal G(x)\,{=}\,Qx$, $Q\,{\in}\, \mathbb{R}^{n\times n}$, and assume $\cal G$ satisfies Assumption \ref{assump:regularity_Helmholtz}. Suppose that: i) all off-diagonal entries of $Q_a$ are nonzero, ii) $Q_s$ satisfies the  assumptions in Lemma \ref{lemma:periodic}, and iii) precisely one eigenvalue of $Q_s$ has algebraic multiplicity greater than one. Then, for each $\eta\in(0,1]$, the origin is unstable for the $x$-dynamics in Nesterov's ODE \eqref{eq:initial_Nesterov_ODE0}. \hfill$\square$
\end{theorem}
\begin{proof}
    \submissionOnly{\emph{Step 1 (Definition of the Average System):} By Lemma \ref{lemma:periodic}, there exists $\cal T>0$ such that $\Lambda(\tau, s)$ is periodic in $s$ with period $\cal T$. Indeed, since solutions to \eqref{eq:driftDynamics} are periodic, the matrix exponential $e^{As}$ is periodic in $s$. Because $B(\tau)$ is independent of $s$, the matrix $\Lambda(\tau,s) = e^{-As}B(\tau)e^{As}$ is also periodic with period $\cal T$. \added{By} averaging \eqref{eq:standardLinear} over $s$, we obtain the following slow time-varying\footnote{\added{The system is called slow time-varying because the resulting vector field depends on $\tau$, which evolves slowly since $\frac{d\tau}{ds} = \varepsilon \ll 1$.}} average system:
    {\begin{align}\label{proof:average}
            \frac{d\zeta}{d\added{s}} & =  \varepsilon\left( \overline{B}_1 + \frac{3}{\tau + \added{T_0}}\overline{B}_2\right)\zeta,\qquad \frac{d\tau}{d\added{s}} = \varepsilon,
        \end{align}
        where $\overline{B}_k \coloneqq { \frac{1}{\cal T}\int_{0}^{\cal T}e^{-As}B_ke^{As}}ds,$}
    for $k\in \{1,2\}$, and
    $ B_1 {\coloneqq} \begin{pmatrix}
            0          & 0 \\
            -\hat{Q}_a & 0
        \end{pmatrix}$ and $ B_2{\coloneqq}\begin{pmatrix}
            0 & 0  \\
            0 & -I
        \end{pmatrix}$.\smallbreak}
    \reportOnly{{\color{black}
                \emph{Step 1 (Definition of the Average System):} By Lemma \ref{lemma:periodic}, there exists $\cal T>0$ such that $\Lambda(\tau, s)$ is periodic in $s$ with period $\cal T$. Indeed, since solutions to \eqref{eq:driftDynamics} are periodic, the matrix exponential $e^{As}$ is periodic in $s$. Because $B(\tau)$ is independent of $s$, the matrix $\Lambda(\tau,s) = e^{-As}B(\tau)e^{As}$ is also periodic with period $\cal T$. By averaging system \eqref{eq:standardLinear} over $s$ we otain
                \begin{align*}
                    B(\tau) = \begin{pmatrix}
                                  0          & 0                      \\
                                  -\hat{Q}_a & -\frac{3}{\tau + T_0}I
                              \end{pmatrix}.
                \end{align*}

                We decompose $B(\tau)$ as $B(\tau) = B_1 + \frac{3}{\tau + T_0}B_2$, where
                \begin{align*}
                    B_1 \coloneqq \begin{pmatrix}
                                      0          & 0 \\
                                      -\hat{Q}_a & 0
                                  \end{pmatrix}, \quad
                    B_2 \coloneqq \begin{pmatrix}
                                      0 & 0  \\
                                      0 & -I
                                  \end{pmatrix}.
                \end{align*}
                Averaging over one period $\cal T$ yields the following slow time-varying\footnote{The system is called slow time-varying because the resulting vector field depends on $\tau$, which evolves slowly since $\frac{d\tau}{ds} = \varepsilon \ll 1$.} average system:
                \begin{align}\label{proof:average}
                    \frac{d\phi}{ds} = \varepsilon\left(\overline{B}_1 + \frac{3}{\tau + T_0}\overline{B}_2\right)\phi, \quad \frac{d\tau}{ds} = \varepsilon,
                \end{align}
                where $\overline{B}_k \coloneqq \frac{1}{\cal T}\int_{0}^{\cal T}e^{-As}B_ke^{As}\,ds$ for $k\in\{1,2\}$.\smallbreak}
    }
    \emph{Step 2 (Instability of the Average System):} We analyze the \added{term} $\overline{B}_1$ in \eqref{proof:average} which, as $s\to \infty$, \added{dominates} the stability \added{properties} of \eqref{proof:average}. First, since $\hat{Q}_s$ is  positive definite, \added{by the spectral theorem \cite[2.5.6]{hornMatrixAnalysis2017}}, there exists an orthonormal matrix $P$ such that
    $
        P^\top \hat{Q}_s P = \operatorname{diag}\{q_k \}_{k=1}^n\eqqcolon \Delta_s,~q_k >0,
    $
    where $\{q_k \}_{k=1}^n = \operatorname{spec}(\hat{Q}_s)$, and where $q_k >0$. Let $\hat{P} \coloneqq I_2 \otimes P$, where $\otimes$ denotes the Kronecker product, and define
    $
        \Delta \coloneqq \hat{P}^\top A\hat{P}
        = \begin{pmatrix}
            0         & I \\
            -\Delta_s & 0
        \end{pmatrix}.
    $
    Then, we have that
    \begin{align}\label{proof:matrixExponential}
         & \scalebox{0.95}{$\hat{P}^\top e^{\pm A\tau}\hat{P}=e^{\Delta\tau}
                = \begin{pmatrix}
                      \mathcal{C}(\added{s})      & \pm\mathcal{S}_2(\added{s}) \\
                      \mp\mathcal{S}_1(\added{s}) & \mathcal{C}(\added{s})
                  \end{pmatrix}$},                       
        %
    \end{align}
    \noindent with $\mathcal{C}(\added{s}) \!\coloneqq\! \operatorname{diag}\left\{\cos(\lambda_k\added{s})\right\}_{k=1}^n$,
    $\mathcal{S}_1(\added{s})\!\coloneqq\! \operatorname{diag}\left\{\lambda_k\sin(\lambda_k\added{s})\right\}_{k=1}^n$, $\mathcal{S}_2(\added{s})\!\coloneqq\! \operatorname{diag}\left\{\lambda_k^{\!-1}\sin(\lambda_k\added{s})\right\}_{k=1}^n$, and $\lambda_k\coloneqq \sqrt{q_k }$ for all $k=\{1,2,\ldots,n\}$.
    Using these definitions, we analyze the term
    $
        \hat{P}^\top \overline{B}_1 \hat{P} = \frac{1}{\cal T}\int_{0}^{\cal T}e^{-\Delta \added{s}}\hat{P}^\top B_1\hat{P}e^{\Delta \added{s}}d\added{s}.
    $
    From \eqref{proof:matrixExponential}, by letting $\tilde{Q}_a\coloneqq P^\top \hat{Q}_a P$, we have that
    \begin{align*}
        \scalebox{0.985}{$
                e^{-\Delta \added{s}}\hat{P}^\top B_1\hat{P}e^{\Delta \added{s}}=  \begin{pmatrix}
                                                                                       \mathcal{S}_2(\added{s})\tilde{Q}_a \mathcal{C}(\added{s}) & \mathcal{S}_2(\added{s})\tilde{Q}_a\mathcal{S}_2(\added{s}) \\
                                                                                       -\mathcal{C}(\added{s})\tilde{Q}_a\mathcal{C}(\added{s})   & -\mathcal{C}(\added{s})\tilde{Q}_a\mathcal{S}_2(\added{s})
                                                                                   \end{pmatrix}.$}
    \end{align*}
    Then, from the definitions of $C(\added{s})$ and $\mathcal{S}_2(\added{s})$:
    {\small
    \begin{align*}
        \left[\mathcal{S}_2(\added{s})\tilde{Q}_a \mathcal{C}(\added{s})\right]_{ij}   & = \lambda_{i}^{-1}\sin(\lambda_i\added{s})\tilde{Q}_{a,ij}\cos(\lambda_j \added{s}) \\
        \left[\mathcal{S}_2(\added{s})\tilde{Q}_a \mathcal{S}_2(\added{s})\right]_{ij} & = \lambda_{i}^{-2}\sin(\lambda_i\added{s})\tilde{Q}_{a,ij}\sin(\lambda_j \added{s}) \\
        \left[\mathcal{C}(\added{s})\tilde{Q}_a \mathcal{C}(\added{s})\right]_{ij}     & = \cos(\lambda_i\added{s})\tilde{Q}_{a,ij}\cos(\lambda_j \added{s})                 \\
        \left[\mathcal{C}(\added{s})\tilde{Q}_a \mathcal{S}_2(\added{s})\right]_{ij}   & = \lambda_j^{-1}\cos(\lambda_i\added{s})\tilde{Q}_{a,ij}\sin(\lambda_j \added{s}).
    \end{align*}}
    Since the period $\cal T$ can be written as the least common multiple of $\left\{\!\frac{1}{\lambda_1},\ldots, \frac{1}{\lambda_n}\!\right\}$, it follows that the terms $\mathcal{C}(\added{s})\tilde{Q}_a\mathcal{S}_2(\added{s})$ and $\mathcal{C}(\added{s})\tilde{Q}_{a}\mathcal{S}_2(\added{s})$ vanish under the averaging operation.
    Therefore, \scalebox{0.95}{$
            \hat{P}^\top \overline{B}_1 \hat{P} = \frac{1}{2}\begin{pmatrix}
                0                    & \overline{Q}_a^{(2)} \\
                \overline{Q}_a^{(1)} & 0
            \end{pmatrix},
        $}
    where $\left(\overline{Q}_a^{(1)}\right)_{ij}\coloneqq -(\overline{Q}_a)_{ij}$, $\left(\overline{Q}_a^{(2)}\right)_{ij}\coloneqq (\overline{Q}_a)_{ij}/q_i$, and
    \begin{align}\label{proof:vanishingConditions}
        \left(\overline{Q}_a\right)_{ij} & \coloneqq\begin{cases}
                                                        0\quad           & \text{if }i=j~\lor~q_i\neq q_j \\
                                                        \tilde{Q}_{a,ij} & \text{if }q_i =q_j
                                                    \end{cases},
    \end{align}
    %
    %
    where we used the fact that $\tilde{Q}_a^\top=-\tilde{Q}_a$ which implies that $\tilde{Q}_{a,ii}=0$ for all $i\in\{1,2,\ldots,n\}$.

    Given that $Q_s$ has exactly one degenerate eigenvalue, there exists $q>0$ such that $q_i=q$ for all indices $i$ in some subset $\mathcal{I}\subset \{1,2,\ldots,n\}$. Also, for all $(j,k)\in\mathcal{J}\coloneqq\{1,2,\ldots,n\}\setminus \mathcal{I}$, we have $q_j\neq q_k\neq q$. Thus, from \eqref{proof:vanishingConditions}, both $\overline{Q}_a^{(1)}$ and $\overline{Q}_a^{(2)}$ have zeros in their $j^{th}$ row and $j^{th}$ column for all $j\in \mathcal{J}$. Hence, given $\lambda \in \mathbb{C}$, we have that
    $
        \det(\lambda I -\hat{P}^\top \overline{B}_1 \hat{P}) = \lambda^{|\mathcal{J}|}\det\left(\lambda^2 I+\frac{1}{2q}\tilde{Q}^{2}_{\mathcal{J}}\right),
    $
    where $|\mathcal{J}|$ denotes the cardinality of $\mathcal{J}$, and where $\tilde{Q}_{\mathcal{J}}$ is obtained from $\tilde{Q}_a$ by removing its $j^{th}$ row and column for every $j\in \mathcal{J}$. Therefore,
    \begin{equation}\label{proof:eigenvaluesPrefinal}
        \scalebox{0.95}{$\left\{\lambda = \pm \sqrt{\frac{-\mu^2}{2q}}:~\mu\in\operatorname{spec}(\tilde{Q}_{\mathcal{J}})\right\}\subset \operatorname{spec}(\overline{B}_1)$},
    \end{equation}
    where we used the fact that the spectrum of a matrix is invariant under similarity transformations.

    Now, since $ \tilde{Q}_a$ is skew-symmetric, $\tilde{Q}_{\mathcal{J}}$  is also skew-symmetric by construction. Additionally, given that $ Q_{a,ij}\neq 0$ for all $i,j\in\{1,2,\dots,n\}$ by assumption, it follows that there exists a set  $\{\omega_k\}_{k=1}^{m}\subset\mathbb{R}_{>0}$, with $m=\left\lfloor|\mathcal{J}|/2\right\rfloor$, such that
    $
        \left\{\pm i\,\omega_k\right\}_{k=1}^m\subset\operatorname{spec}\left(\tilde{Q}_{\mathcal{J}}\right).
    $
    %
    \added{Thus, letting $\mu = i\omega_k$ in \eqref{proof:eigenvaluesPrefinal} yields $\lambda = \pm\omega_k/\sqrt{2q}$ and $\lambda\in \mathrm{spec}(\overline{B}_1)$, giving a positive real eigenvalue for $\overline{B}_1$}. Hence, the origin is unstable under the $\phi$ dynamics of \eqref{proof:average}.

    \emph{Step 3 (Instability of the Original System)}: Instability of system \eqref{eq:standardLinear} follows from Step 2 by applying the Floquet Theorem in \cite[Sec. 2]{sarychevLieChronologicoalgebraic2001}. Also, by Theorem \ref{variationOfConstants} we note that
    $|y(s)|=|\Phi^h_s(z)|=|e^{As} z(s)|\ge \sigma_{\min}(e^{As})|z(s)|$
    for all $s\in\mathbb{R}_{\ge0}$,
    where $\sigma_{\min}(\Lambda)$ denotes the minimum singular value of $\Lambda\in\R^{n\times n}$. Since $\sigma_{\min}(e^{As})>0$ by the same reasoning used in the proof of Lemma \ref{lemma:periodic}, the instability of system \eqref{eq:preStandardLinear} is implied by the instability of system \eqref{eq:standardLinear}.
    %
\end{proof}
\begin{remark}
    Theorem \ref{maintheorem} provides a theoretical explanation for the instability observed \added{numerically} in \cite{ochoaMomentumBasedNashSetSeeking2024} whenever non-conservative terms appear in the vector field $\cal G$ in \eqref{eq:initial_Nesterov_ODE0}. Such terms arise when \added{$\cal G$ is the pseudo-gradient of a non-potential game}, or in consensus-based dynamics when the communication graph is directed \cite{ochoaDecentralizedConcurrentLearning2024}. Note that the instability emerges \emph{regardless of the size} of $Q_a$. \added{In this sense}, Theorem \ref{maintheorem} also provides a procedure for the synthesis of arbitrarily small (compared to the conservative part of $\cal G$), adversarial, state-dependent perturbations able to destabilize \eqref{eq:initial_Nesterov_ODE0} under conservative mappings of the form $\cal G=\nabla J$, by adding to $\nabla J$ a perturbation of the form $\rot K$. \hfill$\square$
\end{remark}
\section{Precluding Instability via Restarting}\label{sec:stabilityRestarting}
To address the instability of Nesterov's ODE under non-conservative mappings, we regularize the dynamics using resets that restart the momentum. 
Similar approaches have been studied in \cite{teel2019first} under conservative maps, in \cite{ochoaMomentumBasedNashSetSeeking2024} for Nash equilibrium-seeking problems in monotone games, and in \cite{ochoaDecentralizedConcurrentLearning2024} for consensus-based concurrent learning in directed-graphs. Compared to these \added{works}, we establish improved convergence rates and tighter resetting conditions for a subclass of vector fields satisfying Assumption~\ref{assump:regularity_Helmholtz}-(i), (ii), \added{that includes nonlinear vector fields beyond the linear setting studied in Section \ref{sec:instability}.}

\added{Starting from Nesterov's ODE in \eqref{eq:initial_Nesterov_ODE0}, we let $q\coloneqq x$, $p\coloneqq\dot{x}$, and embed the dynamics in the HDS $\cal H=(C,f,D,g)$ with state $\chi = (q,p,\tau) \in \mathbb{R}^n \times \mathbb{R}^n \times \mathbb{R}_{\geq 0}$, and data
    \begin{align*}
        f(\chi) & \coloneqq \begin{pmatrix}
                                p                            \\
                                -\frac{3}{\tau}p - \cal G(q) \\
                                \eta
                            \end{pmatrix}, \quad
        g(\chi) \coloneqq \begin{pmatrix}q\\0\\T_0\end{pmatrix},                                       \\
        C       & \coloneqq \{(q,p,\tau) : q \in \mathbb{R}^n, p \in \mathbb{R}^n, \tau \in [T_0,T]\}, \\
        D       & \coloneqq \{(q,p,\tau) : q \in \mathbb{R}^n, p \in \mathbb{R}^n, \tau = T\},
    \end{align*}
    where $\eta \in (0,1]$ and $0 < T_0 < T$ are tunable parameters. The HDS periodically resets $\tau$ to $T_0$ when it reaches $T$, preventing the damping term $-3p/\tau $ from vanishing, and $p$ to $0$ to ensure the strict decrease of a suitable Lyapunov function across jumps.}
For this HDS, we analyze the stability of the compact set \added{$\cal A \coloneqq \{x^\star\} \times \{0\} \times [T_0,T]$}, where $x^\star \in \mathbb{R}^n$ is the unique point\footnote{Guaranteed to exist due to the strong monotonicity properties of $\cal G$.} that satisfies $\cal G(x^\star)=0$.
    {\color{black}
        \begin{theorem}\label{thm:hybrid_stability}
            Let $\cal G$ satisfy Assumption~\ref{assump:regularity_Helmholtz}-(i),(ii). Suppose that $\eta\in(0,1)$, and $\underline{T} < T \leq \overline{T}$, where $\underline{T}^2 \coloneqq T_0^2 + 4\eta^2/\kappa_J$ and $\overline{T} \coloneqq 2\min\{3(1-\eta), \kappa_J\eta\}/\ell_K$. Then, the set $\cal A$ is UGES for $\cal H$. Moreover, there exists $\rho > 0$ such that, for each compact set $K_0 \subset C\cup D$, there exist $M_J, M_{\cal G} > 0$ for which each solution $\chi = (q,p,\tau)$ with $\chi(0,0) \in K_0$ satisfies
            \begin{align*}
                \tilde{J}(q(t,j)) \leq \frac{M_J T^2 e^{-\rho j}}{\tau^2(t,j)}, \quad |\cal G(q(t,j))|^2 \leq \frac{M_{\cal G} T^2 e^{-\rho j}}{\tau^2(t,j)}
            \end{align*}
            $\forall (t,j) \in \dom \chi$, where $\tilde{J}(q) \coloneqq J(q) \!-\! J(x^\star)~\forall q \in \mathbb{R}^n$.\hfill$\square$
        \end{theorem}
    }
    {\color{black}
        \submissionOnly{
            \begin{proof}
                For each $\chi\in C\cup D$, consider the Lyapunov function
                \begin{align*}
                    V(\chi) = a\left|q + \frac{\tau}{b}p - x^\star\right|^2 + c\tau^2|p|^2 + \delta\tau^2(J(q) - J(x^\star)),
                \end{align*}
                where $a = 2\eta b/T^2$, $b \coloneqq 3-\eta$, $c \coloneqq \frac{3a(1-\eta)}{2\eta b^2}$, and $\delta \coloneqq \frac{a}{\eta b} = \frac{2}{T^2}$. Since $\tau \in [T_0, T]$, and $J$ is strongly convex by Assumption \ref{assump:regularity_Helmholtz}, we obtain $\underline{c}|\chi|_{\cal A}^2 \leq V(\chi) \leq \overline{c}|\chi|_{\cal A}^2$ where $\underline{c} \coloneqq T_0^2\min\{c, \frac{\delta\kappa_J}{2}\}$, $\overline{c} \coloneqq \max\{a + \frac{aT}{b} + \frac{\delta T^2\ell_J}{2}, mT^2 + \frac{aT}{b}\}$, and $m \coloneqq \frac{a}{b^2} + c = \delta/2$.

                Computing $\dot{V}(\chi) \!=\! \langle \nabla V(\chi), f(\chi) \rangle$, and using the Helmholtz decomposition $\cal G(q) = \nabla J(q) + \rot K(q)$, yields
                \begin{align*}
                     & \dot{V}(\chi) = -2c\tau(3-\eta)|p|^2                                                                               \\
                     & ~~ - \frac{2a\tau}{b}\left(\inner{ q - x^\star}{\nabla J(q)} - (J(q) - J(x^\star))\right)                          \\
                     & \quad - \frac{2a\tau}{b}\inner{q - x^\star}{\rot K(q)} - \delta\tau^2\inner{p}{\rot K(q)}\quad \forall \chi \in C.
                \end{align*}
                By Assumption \ref{assump:regularity_Helmholtz}, $\langle \nabla J(q), q - x^\star \rangle - (J(q) - J(x^\star)) \geq \frac{\kappa_J}{2}|q - x^\star|^2$, $\langle q - x^\star, \rot K(q) \rangle \geq 0$, and $|\rot K(q)| \leq \ell_K|q - x^\star|$ for all $q\in\R^n$. Using these facts, together with the Cauchy-Schwarz and Young's inequalities, we obtain that
                \begin{align*}
                    \dot{V}(\chi) & \leq -\tau\lambda(|p|^2 + |q - x^\star|^2) + \frac{\delta T\tau\ell_K}{2}(|p|^2 + |q - x^\star|^2)                \\
                                  & = -\lambda\tau\left(1 - \frac{T}{\overline{T}}\right)|\chi|_{\cal A}^2 \leq -\mu V(\chi)\quad \forall \chi \in C,
                \end{align*}
                where $\lambda \coloneqq \min\{2c(3-\eta), \frac{a\kappa_J}{b}\}$ and $\mu \coloneqq \lambda(\overline{T}-T)T_0/\overline{T}\overline{c}$.

                On the other hand, for $\chi \in D$, $V(g(\chi)) - V(\chi) = -mT^2|p|^2 - \frac{2aT}{b}\langle q - x^\star, p \rangle - \delta(T^2 - T_0^2)(J(q) - J(x^\star))$. Defining $\Gamma \coloneqq \sqrt{(T^2 - T_0^2)\kappa_J}$, and applying Cauchy-Schwarz together with Young's inequality with parameter $\theta = T/\Gamma$, and using the strong convexity of $J$, by letting $\Delta V(\chi)\coloneqq V(g(\chi))-V(\chi)$, we obtain that
                \begin{align*}
                    \Delta V(\chi) & \!\leq\! -\left(\frac{aT^2}{2\eta b} - \frac{aT^2}{b\Gamma}\right)|p|^2 - \left(\frac{a\Gamma^2}{2\eta b} - \frac{a\Gamma}{b}\right)|q - x^\star|^2 \\
                                   & = -\nu_1|p|^2 - \nu_2|q - x^\star|^2 \leq -\frac{\nu}{\overline{c}}V(\chi)
                \end{align*}
                for all $\chi\in D$, where $\nu_1 \coloneqq 1 - \frac{2\eta}{\Gamma}$, $\nu_2 \coloneqq \frac{\Gamma(\Gamma - 2\eta)}{T^2}$, and $\nu \coloneqq \min\{\nu_1, \nu_2\}$. Since $T > \underline{T}$ implies $\Gamma > 2\eta$, $\nu\in(0,1)$.

                The quadratic bounds on $V$, combined with the strict decrease of $V$ during the flows and jumps of $\cal H$ establish that $\cal A$ is UGES for $\cal H$ via \cite[Thm. 1]{teelLyapunovBasedSufficient2013}. To obtain the convergence bounds for $\cal G$ and $\tilde{J}$, let $K_0 \subset C \cup D$ be compact. For any solution $\chi$ to $\cal H$ with $\chi(0,0) \in K_0$, define $c_j \coloneqq V(\chi(s_j, j))$ where $s_j \coloneqq \min\{t : (t,j) \in\dom\chi\}$. The flow and jump conditions give $c_{j+1} \leq e^{-\rho}c_j$ where $\rho \coloneqq -\ln(1 - \frac{\nu}{\overline{c}}) + \mu(T - T_0) > 0$. Thus, $c_j \leq e^{-\rho j}V(\chi(0,0))$. From $\delta\tau^2(J(q) - J(x^\star)) \leq V(\chi)$, we obtain $\tilde{J}(q(t,j)) \leq V(\chi(0,0))T^2 e^{-\rho j}/2\tau^2(t,j)$ for all $(t,j) \in\dom\chi$. Using strong convexity and Lipschitz continuity of $\cal G$ gives $|\cal G(q(t,j))|^2 \leq (\ell_J + \ell_K)^2 V(\chi(0,0)) T^2 e^{-\rho j}\kappa_J\tau^2(t,j)$. Letting $M_J \coloneqq \frac{1}{2}\max_{\chi \in K_0} V(\chi)$ and $M_{\cal G} \coloneqq \frac{2(\ell_J + \ell_K)^2}{\kappa_J}M_J$ completes the proof. Additional details are presented in the extended manuscript \cite{ochoa2025NesterovInstability}.
            \end{proof}}
        \reportOnly{
            Consider the Lyapunov function
            \begin{align}\label{eq:LyapunovDef}
                V(\chi) & = a\left|q + \frac{\tau}{b}p - x^\star\right|^2 + c\tau^2|p|^2 + \delta\tau^2(J(q) - J(x^\star)),
            \end{align}
            where $a=2\eta b/T^2$, $b\coloneqq 3-\eta$,  $c \coloneqq \frac{3a(1-\eta)}{2\eta b^2}$, and $\delta \coloneqq \frac{a}{\eta b}$, have been chosen to ensure suitable cancellations in the forthcoming analysis. To simplify computations we note that $V$ can be equivalently written as follows:
            \begin{align*}
                V(\chi) & = a|q - x^\star|^2 + \frac{2a\tau}{b}\langle q - x^\star, p \rangle                \\
                        & \qquad + \frac{a\tau^2}{b^2}|p|^2 + c\tau^2|p|^2 + \delta\tau^2(J(q) - J(x^\star)) \\
                        & = a|q - x^\star|^2 + \frac{2a\tau}{b}\langle q - x^\star, p \rangle                \\
                        & \qquad + m\tau^2|p|^2 + \delta\tau^2(J(q) - J(x^\star)),
            \end{align*}
            where $m\coloneqq  \frac{a}{ b^2} + c =\delta/2$.

            From \eqref{eq:LyapunovDef}, noting that $a\left|q + \frac{\tau}{b}p - x^\star\right|^2 \geq 0$,  $\tau \geq T_0$ for all $(q,p,\tau)\in C\cup D$, and that $J(q) - J(x^\star) \geq \frac{\kappa_J}{2}|q - x^\star|^2$ from Assumption \ref{assump:regularity_Helmholtz}, we obtain that $V(\chi)\ge \underline{c}|\chi|_{\cal A}^2$, where
            $\underline{c}\coloneqq T_0^2\min\left\{c, \frac{\delta\kappa_J}{2}\right\}$.

            On the other hand, using Cauchy-Schwarz and Young's inequalities gives $2\inner{q - x^\star}{p} \leq |q - x^\star|^2 + |p|^2$. Thus,
            \begin{align*}
                V(\chi) & \leq a|q - x^\star|^2 + \frac{aT}{b}(|q - x^\star|^2 + |p|^2) + mT^2|p|^2 + \\
                        & \quad  \frac{\delta T^2\ell_J}{2}|q - x^\star|^2                            \\
                        & = \left(a + \frac{aT}{b} + \frac{\delta T^2\ell_J}{2}\right)|q - x^\star|^2 \\
                        & \quad+ \left(mT^2 + \frac{a T}{b}\right)|p|^2,                              \\
                        & \le \overline{c}|\chi|^2_{\cal A},
            \end{align*}
            where $\overline{c} \coloneqq \max\left\{a + \frac{aT}{b} + \frac{\delta T^2\ell_J}{2}, mT^2 + \frac{aT}{b}\right\}$, and where we used the fact that $\tau \leq T$ for all $(q,p,\tau)\in C\cup D$, and that $J(q) - J(x^\star) \leq \frac{\ell_J}{2}|q - x^\star|^2$ for all $q\in \R^n$ by Assumption \ref{assump:regularity_Helmholtz}.

            Now, for each $\chi\in \R^n$, we compute the Lie derivative for each term in $V$ along $f$ as follows:
            \begin{align*}
                 & \inner{\nabla(a|q - x^\star|^2)}{f(\chi)}
                = 2a\inner{q - x^\star}{p},                                                                            \\[1ex]
                 & \big\langle\nabla \left(\tfrac{2a\tau}{b}\inner{q - x^\star}{p}\right), f(\chi)\big\rangle          \\
                 & \qquad= \tfrac{2a\eta}{b}\inner{q - x^\star}{p} + \tfrac{2a\tau}{b}|p|^2                            \\
                 & \qquad\quad + \tfrac{2a\tau}{b}\inner{q - x^\star}{-\tfrac{3}{\tau}p - \cal G(q)}                   \\
                 & \qquad= \tfrac{2a\eta}{b}\inner{q - x^\star}{p} + \tfrac{2a\tau}{b}|p|^2                            \\
                 & \qquad\quad - \tfrac{6a}{b}\inner{q - x^\star}{p} - \tfrac{2a\tau}{b}\inner{q - x^\star}{\cal G(q)} \\
                 & \qquad= \tfrac{2a}{b}\left(\eta-3\right) \inner{q - x^\star}{p} + \tfrac{2a\tau}{b}|p|^2            \\
                 & \qquad\quad - \tfrac{2a\tau}{b}\inner{q - x^\star}{\cal G(q)},                                      \\[1ex]
                 & \inner{\nabla(m\tau^2|p|^2)}{f(\chi)}                                                               \\
                 & \qquad= 2m\tau\eta|p|^2 + 2m\tau^2\inner{p}{-\tfrac{3}{\tau}p - \cal G(q)}                          \\
                 & \qquad= 2m\tau\eta|p|^2 - 6m\tau|p|^2 - 2m\tau^2\inner{p}{\cal G(q)}                                \\
                 & \qquad= 2m\tau\left(\eta-3\right)|p|^2 - 2m\tau^2\inner{p}{\cal G(q)},                              \\[1ex]
                 & \inner{\nabla(\delta\tau^2(J(q) - J(x^\star)))}{f(\chi)}                                            \\
                 & \qquad= 2\delta\tau\eta(J(q) - J(x^\star)) + \delta\tau^2\inner{\nabla J(q)}{p}.
            \end{align*}
            Therefore, letting $\dot{V}(\chi)\coloneqq  \inner{\nabla V(\chi)}{f(\chi)}$ for all $\chi\in \R^n$, we get that
            \begin{align*}
                \dot{V}(\chi) & = \left(2a + \frac{2a}{b}\left(\eta-3\right)\right)\inner{q-x^\star}{p} + \\
                            &\qquad +  2\tau\left(\frac{a}{ b} + m(\eta - 3)\right)|p|^2 \\
                              & \quad - \frac{2a\tau}{ b}\inner{q - x^\star}{\cal G(q)}- 2m\tau^2\langle p, \cal G(q) \rangle                                 \\
                              & \qquad  + 2\delta\tau\eta(J(q) - J(x^\star)) + \delta\tau^2\langle \nabla J(q), p \rangle.
            \end{align*}
            Since $b=3-\eta$, we obtain that $\frac{a}{b} + m(\eta-3)=\frac{a}{3-\eta} + \left(\frac{a}{(3-\eta)^2} + c\right)(\eta-3) = -c(3-\eta)$. Thus, from the above expression,
            \begin{align*}
                \dot{V}(\chi) & =  -2c\tau\left(3-\eta\right)|p|^2  - \frac{2a\tau}{ b}\inner{q - x^\star}{\cal G(q)} - 2m\tau^2\langle p, \cal G(q) \rangle \\
                              & \qquad+ 2\delta\tau\eta(J(q) - J(x^\star)) + \delta\tau^2\langle \nabla J(q), p \rangle.
            \end{align*}
            Using the fact that $\cal G(q) = \nabla J(q) + \rot  K(q)$ for all $q\in\R^n$ via Lemma \ref{lemma:helmholtzDecomposition}, we have that
            \begin{align}
                \dot{V}(\chi) & = -2c\tau(3-\eta)|p|^2 - \frac{2a\tau}{ b}\inner{q - x^\star}{\nabla J(q)}\notag             \\
                              & \quad- \frac{2a\tau}{ b}\inner{q - x^\star}{\rot  K(q)} \notag                               \\
                              & \qquad - 2m\tau^2\inner{p}{\nabla J(q)} - 2m\tau^2\inner{p}{\rot  K(q)} \notag               \\
                              & \qquad \quad+ 2\delta\tau\eta(J(q) - J(x^\star)) + \delta\tau^2\inner{\nabla J(q)}{p} \notag \\
                              & = -2c\tau(3-\eta)|p|^2 - \frac{2a\tau}{ b}\inner{q - x^\star}{\nabla J(q)}\notag             \\
                              & \quad - \frac{2a\tau}{ b}\inner{q - x^\star}{\rot  K(q)} \notag                              \\
                              & \qquad + (\delta - 2m)\tau^2\inner{p}{\nabla J(q)} - 2m\tau^2\inner{p}{\rot  K(q)} \notag    \\
                              & \qquad \quad+ 2\delta\tau\eta(J(q) - J(x^\star)).\label{eq:LyapunovFinal:aux0}
            \end{align}
            Given that $m=\frac{a}{b^2} + c=\frac{\delta}{2}$, and $\delta=\frac{a}{\eta b}$, from \eqref{eq:LyapunovFinal:aux0} it follows that
            \begin{align*}
                \dot{V}(\chi) & = -2c\tau(3-\eta)|p|^2                                                                                     \\
                              & \qquad  - \frac{2a\tau}{ b}\left[\inner{q - x^\star}{\nabla J(q)} - (J(q) - J(x^\star))\right]             \\
                              & \qquad\quad  - \frac{2a\tau}{ b}\inner{q - x^\star}{\rot  K(q)} - \delta\tau^2\inner{p}{\rot  K(q)}\notag.
            \end{align*}
            By Assumption \ref{assump:regularity_Helmholtz}, we have $\inner{\nabla J(q)}{q - x^\star} - (J(q) - J(x^\star)) \geq \frac{\kappa_J}{2}|q - x^\star|^2$ and $\inner{q - x^\star}{\rot  K(q)}\geq 0$. Thus, for each $\chi\in C$, we obtain
            \begin{align*}
                \dot{V}(\chi) & \leq -2c\tau(3-\eta)|p|^2 - \frac{a}{b}\tau\kappa_J|q - x^\star|^2                        \\
                              & \qquad - \delta\tau^2\inner{p}{\rot  K(q)}                                                \\
                              & \le -\tau\lambda\left(|p|^2 + |q - x^\star|^2\right) - \delta\tau^2\inner{p}{\rot  K(q)},
            \end{align*}
            where $\lambda \coloneqq \min\left\{2c(3-\eta), \frac{a\kappa_J}{b}\right\}$.
            Since $\rot K(x^\star) = 0$ and $|\rot K(q)| \leq \ell_K |q - x^\star|$ by Assumption \ref{assump:regularity_Helmholtz}(ii) and the fact that $\cal G(x^\star)=\nabla J(x^\star)=0$, using Cauchy-Schwarz and $\tau \leq T$, we have $-\delta\tau^2\inner{p}{\rot K(q)} \leq \delta T\tau \ell_K |p||q - x^\star|$. Applying Young's inequality
            \begin{align*}
                \dot{V}(\chi) & \leq -\tau\lambda\left(|p|^2 + |q - x^\star|^2\right) + \frac{\delta T\tau\ell_K}{2}\left(|p|^2 + |q - x^\star|^2\right) \\
                              & = -\tau\left(\lambda - \frac{\delta T\ell_K}{2}\right)\left(|p|^2 + |q - x^\star|^2\right)                               \\
                              & = -\lambda\tau \left(1-\frac{T}{\overline{T}}\right)|\chi|_{\cal A}^2                                                    \\
                              & \le -\mu V(\chi)
            \end{align*}
            for all $\chi\in C$, and where $\mu\coloneqq \lambda (\overline{T}-T)T_0/\overline{T}\overline{c}>0$.

            On the other hand, for all $\chi \in D$ we have that
            \begin{align*}
                 & V(g(\chi)) - V(\chi)  = V(q, 0, T_0) - V(q, p, T)                                               \\
                 & = \left(a|q - x^\star|^2 + \frac{2aT_0}{b}\langle q - x^\star, 0 \rangle + mT_0^2|0|^2\right.   \\
                 & \qquad \left. + \delta T_0^2(J(q) - J(x^\star))\right)                                          \\
                 & \quad - \left(a|q - x^\star|^2 + \frac{2aT}{b}\langle q - x^\star, p \rangle + mT^2|p|^2\right. \\
                 & \left. + \delta T^2(J(q) - J(x^\star))\right)                                                   \\
                 & = -mT^2|p|^2 - \frac{2aT}{b}\langle q - x^\star, p \rangle                                      \\
                 & \quad- \delta(T^2 - T_0^2)(J(q) - J(x^\star)).
            \end{align*}
            Define $\Gamma \coloneqq \sqrt{(T^2 - T_0^2)\kappa_J}$ and $\theta\coloneqq T/\Gamma$. Using Cauchy-Schwarz inequality, Young's inequality, and Assumption \ref{assump:regularity_Helmholtz}-(i), we obtain that
            \begin{align*}
                 & V(g(\chi)) - V(\chi)  \leq -mT^2|p|^2 + \frac{2aT}{b}|q - x^\star||p| \\
                 &\qquad\qquad\qquad\qquad- \delta(T^2 - T_0^2)(J(q) - J(x^\star)) \notag                        \\
                 & \leq -mT^2|p|^2 + \frac{aT}{b}\left(\frac{|q - x^\star|^2}{\theta} + \theta|p|^2\right) - \frac{\delta\Gamma^2}{2}|q - x^\star|^2            \\
                 & = -\left(\frac{aT^2}{2\eta b} - \frac{aT^2}{b\Gamma}\right)|p|^2 - \left(\frac{a\Gamma^2}{2\eta b} - \frac{a\Gamma}{b}\right)|q - x^\star|^2 \\
                 & = -\frac{aT^2}{2b}\left(\frac{1}{\eta} - \frac{2}{\Gamma}\right)|p|^2 - \frac{a\Gamma}{2\eta b}(\Gamma - 2\eta)|q - x^\star|^2               \\
                 & = -\eta\left(\frac{1}{\eta} - \frac{2}{\Gamma}\right)|p|^2 - \frac{\Gamma}{T^2}(\Gamma - 2\eta)|q - x^\star|^2                               \\
                 & = -\nu_1|p|^2 - \nu_2|q - x^\star|^2\\
                 & \le -\frac{\min\{\nu_1,\nu_2\}}{\overline{c}}V(\chi),
            \end{align*}
            where $\nu_1 \coloneqq 1 - \frac{2\eta}{\Gamma}$, $\nu_2\coloneqq \frac{\Gamma(\Gamma - 2\eta)}{T^2}$, and where we used the fact that $a=2\eta b/T^2$, $m = \frac{a}{2\eta b}$, and $\delta = \frac{a}{\eta b}$. Since $T > \underline{T}$ if and only if $\Gamma > 2\eta$, $\nu\in(0,1)$ by assumption.

            The quadratic bounds on $V$, combined with the strict decrease of $V$ during the flows and jumps of $\cal H$ establish that $\cal A$ is UGES for $\cal H$ via \cite[Thm. 1]{teelLyapunovBasedSufficient2013}. To obtain the convergence bounds for $\cal G$ and $\tilde{J}$, let $K_0 \subset C \cup D$ be compact. For any solution $\chi$ to the hybrid system $\cal H$ with $\chi(0,0) \in K_0$, define $c_j \coloneqq V(\chi(s_j, j))$  for all $j \in \cal J(\chi)\coloneqq \{i\in \N:\exists t\in \R_{\ge0}, (t,j)\in \dom \chi\}$, where $s_j \coloneqq \min\{t \in \mathbb{R}_{\geq 0} : (t,j) \in \dom\chi\}$. From the fact that $\dot{V}(\chi) \leq -\mu V(\chi)$ for all $\chi\in C$ and $V(g(\chi)) \leq (1 - \frac{\nu}{\overline{c}})V(\chi)$ for all $\chi\in D$, we obtain that $c_{j+1} \leq e^{-\mu(T - T_0)}(1 - \tfrac{\nu}{\overline{c}})c_j = e^{-\rho}c_j$, where $\rho \coloneqq \ln(1 - \frac{\nu}{\overline{c}}) + \mu(T - T_0) > 0$. Thus, $c_j \leq e^{-\rho j}V(\chi(0,0))$ for all $j\in \cal J(\chi)$.

            From the fact that $\delta\tau^2(J(q) - J(x^\star)) \leq V(\chi)$ for all $q\in\R^n$, and $\delta = \frac{2}{T^2}$, we obtain that
            \begin{align*}
                \tilde{J}(q(t,j)) \leq \frac{V(\chi(0,0))T^2 e^{-\rho j}}{2\tau^2(t,j)}
            \end{align*}
            for all $(t,j) \in \dom\chi$.

            Now, using the strong convexity of $J$ we have
            \begin{align*}
                |q(t,j) - x^\star|^2 &\leq \frac{2}{\kappa_J}(J(q(t,j)) - J(x^\star)) \\
                &= \frac{2}{\kappa_J}\tilde{J}(q(t,j))\\
                &\leq \frac{V(\chi(0,0))T^2 e^{-\rho j}}{\kappa_J\tau^2(t,j)}.
            \end{align*}
            Since $\cal G(x^\star) = 0$ and $\cal G$ is Lipschitz continuous with constant $\ell_J + \ell_K$ by Assumption \ref{assump:regularity_Helmholtz}, we obtain
            \begin{align*}
                |\cal G(q(t,j))| \leq (\ell_J + \ell_K)|q(t,j) - x^\star|.
            \end{align*}
            Combining these bounds yields
            \begin{align*}
                |\cal G(q(t,j))|^2&\leq (\ell_J + \ell_K)^2|q(t,j) - x^\star|^2 \\
                &\leq \frac{(\ell_J + \ell_K)^2 V(\chi(0,0)) T^2 e^{-\rho j}}{\kappa_J\tau^2(t,j)}
            \end{align*}
            for all $(t,j) \in\dom\chi$. Letting $M_J \coloneqq \frac{1}{2}\max_{\chi \in K_0} V(\chi)$, and $M_{\cal G} \coloneqq \frac{(\ell_J + \ell_K)^2}{\kappa_J}\max_{\chi \in K_0} V(\chi)$ completes the proof.
        }
    }
    \vspace{-5pt}
    {\color{black}
        \begin{corollary}\label{cor:optimalRestart}
            Under the assumptions of Theorem \ref{thm:hybrid_stability}, as $T_0\to0$, $T \in [2\underline{T},e\underline{T}]$ guarantees convergence of $q(t,j)$ to $x^\star$ at a rate of order $\mathcal{O}(e^{-\eta\sqrt{\kappa_J}t})$.\hfill$\square$
        \end{corollary}
        \submissionOnly{
            \begin{proof}
                For any solution $\chi=(q,p,\tau)$ to $\mathcal{H}$, the state $\tau$ resets every $(T-T_0)/\eta$ units of time, giving $j \leq \frac{\eta t}{T - T_0}$ for any $(t,j)\in \dom\chi$. Additionally, from Theorem \ref{thm:hybrid_stability} and strong convexity, we obtain $|q(t,j) - x^\star|^2 \leq k_M\exp(-\alpha(T,T_0)t)$ where $\alpha(T,T_0) = -\frac{\eta}{T - T_0}\ln(1 - \frac{\nu}{\overline{c}})$.

                As $T_0 \to 0$, we have $\nu \to \min\{1, \kappa_J\}(1 - \frac{\underline{T}}{T})$ where $\underline{T} \to \frac{2\eta}{\sqrt{\kappa_J}}$. Setting $\xi = \frac{\underline{T}}{T}$ and $\beta = \frac{\min\{1,\kappa_J\}}{\overline{c}}$, we maximize $f(\xi) = -\xi\ln(1 - \beta(1-\xi))$. The unique critical point satisfies $\ln(1 - \beta(1-\xi^*)) + \frac{\beta\xi^*}{1 - \beta(1-\xi^*)} = 0$, which can be verified to be a maximum by checking $g(\xi) = \ln(1-\beta(1-\xi)) + \frac{\beta\xi}{1-\beta(1-\xi)}$ is strictly increasing with $g(0) < 0$ and $\lim_{\xi \to 1^-} g(\xi) > 0$. For $\beta = 1$, we obtain $\xi^* = e^{-1}$, giving $T^{\text{opt}} = e\underline{T}$. For small $\beta$, $\xi^* \approx 1/2$, giving $T^{\text{opt}} \approx 2\underline{T}$. In general, $T^{\text{opt}} \in [2\underline{T}, e\underline{T}]$.

                At $T=T^{\text{opt}}$, $\alpha(T) = \frac{\sqrt{\kappa_J}}{2} \frac{\beta(\xi^*)^2}{1 - \beta(1-\xi^*)}$, yielding convergence rate $\mathcal{O}(e^{-c\eta\sqrt{\kappa_J}t})$ for $c > 0$. Additional details are presented in the the extended manuscript \cite{ochoa2025NesterovInstability}.
            \end{proof}}
        \reportOnly{\begin{proof}
                For each solution $\chi=(q,p,\tau)$ to $\mathcal{H}$, we have that $j \leq \frac{\eta t}{T - T_0}$ for all $(t,j) \in \text{dom}\chi$, by the periodicity of $\tau$.

                Now, from the proof of Theorem \ref{thm:hybrid_stability}, for each compact set $K_0\subset C\cup D$, and each solution $\chi$ to $\cal H$ with $\chi(0,0)\in K_0$, there exists $M_J > 0$ such that $\tilde{J}(q(t,j)) \leq \frac{M_J T^2 e^{-\rho j}}{\tau^2(t,j)}$ where $\rho = -\ln(1 - \frac{\nu}{\overline{c}}) + \mu(T - T_0) > 0$. Using $\tau(t,j) \geq T_0$, the strong convexity of $J$ from Assumption \ref{assump:regularity_Helmholtz}, and the fact that $e^{-\rho j} = (1 - \frac{\nu}{\overline{c}})^j e^{-j\mu(T - T_0)} \leq (1 - \frac{\nu}{\overline{c}})^j$, we obtain
                \begin{align*}
                    |q(t,j) - x^\star|^2 \leq \frac{2M_J T^2}{\kappa_J T_0^2} \left(1 - \frac{\nu}{\overline{c}}\right)^j.
                \end{align*}
                Using $(1-x)^a = e^{a\ln(1-x)}$ and $j \leq \frac{\eta t}{T - T_0}$,
                \begin{align}\label{eq:distance_time}
                    |q(t,j) - x^\star|^2 \leq k_M\exp\left(-\hat{\alpha}(T,T_0)t\right),
                \end{align}
                where $k_M \coloneqq \frac{2M_J T^2}{\kappa_J T_0^2}$ and
                $\hat{\alpha}(T,T_0) \coloneqq -\frac{\eta}{T - T_0}\ln\left(1 - \frac{\nu}{\overline{c}}\right)$.

                Now, recall that $\nu = \min\{\nu_1, \nu_2\}$ where $\nu_1 = 1 - \tfrac{2\eta}{\Gamma}$, $\nu_2 = \tfrac{\Gamma(\Gamma - 2\eta)}{T^2}$, and $\Gamma = \sqrt{(T^2 - T_0^2)\kappa_J}$. As $T_0 \to 0$, we have that $\Gamma \to T\sqrt{\kappa_J}$, which yields
                \begin{align*}
                    \nu \to \min\{1, \kappa_J\}\left(1 - \frac{\underline{T}}{T}\right),
                \end{align*}
                where we used the fact that
                $\underline{T} = \sqrt{T_0^2 + \frac{4\eta^2}{\kappa_J}} \to \frac{2\eta}{\sqrt{\kappa_J}}$.
                Substituting into $\hat{\alpha}(T,T_0)$, with $T_0 \to 0$, gives
                \begin{align}\label{eq:alpha_limit}
                    \hat{\alpha}(T,T_0)\to\alpha(T)
                    \coloneqq -\frac{\eta}{T}\ln\left(1 - \beta\left(1-\frac{\underline{T}}{T}\right)\right),
                \end{align}
                where $\beta \coloneqq \frac{\min\{1,\kappa_J\}}{\overline{c}}$.

                To find the optimal restart time $T$ that maximizes $\alpha(T)$, we introduce the dimensionless variable $\xi \coloneqq \frac{\underline{T}}{T} \in (0,1)$, so $T = \frac{\underline{T}}{\xi}$. From \eqref{eq:alpha_limit}, we define
                $\tilde{\alpha}(\xi) \coloneqq \alpha(\underline{T}/\xi) = -\frac{\eta}{\underline{T}}\xi\ln(1 - \beta(1-\xi))$.
                Since $\frac{\eta}{\underline{T}} = \frac{\sqrt{\kappa_J}}{2}$ is a positive constant, maximizing $\alpha(\xi)$ is equivalent to maximizing
                \begin{align*}
                    \theta(\xi) \coloneqq -\xi\ln(1 - \beta(1-\xi)),~\xi\in(0,1).
                \end{align*}
                By Lemma \ref{lem:unique_optimum}, $\theta$ has a unique maximizer $\xi^{\text{opt}} \in (0,1)$ satisfying
                $
                    \ln(1 - \beta(1-\xi^{\text{opt}})) + \frac{\beta\xi^{\text{opt}}}{1 - \beta(1-\xi^{\text{opt}})} = 0.
                $
                Thus, letting $T=T^{\text{opt}}\coloneqq \underline{T}/\xi^{\text{opt}}$, gives
                \begin{align}\label{eq:alpha_final}
                    \alpha(T^{\text{opt}})
                    = \frac{\sqrt{\kappa_J}}{2} \cdot \frac{\beta(\xi^{\text{opt}})^2}{1 - \beta(1-\xi^{\text{opt}})}.
                \end{align}
                Defining $c \coloneqq \frac{\beta(\xi^{\text{opt}})^2}{2(1 - \beta(1-\xi^{\text{opt}}))} > 0$, from \eqref{eq:distance_time}, we obtain that
                \begin{align*}
                    |q(t,j) - x^\star| = \mathcal{O}(e^{-c\eta\sqrt{\kappa_J}t})
                \end{align*}
                which obtains the convergence bound, and where $c$ depends on $\overline{c}$ and $\kappa_J$ through $\beta$ and $\xi^{\text{opt}}$.

                Now, to obtain an expression for $T^{\text{opt}}$ we analyze by cases the behavior of $\xi^{\text{opt}}$.

                \textit{Case $\beta = 1$.} When $\overline{c} = \min\{1,\kappa_J\}$, equation    $
                    \ln(1 - \beta(1-\xi^{\text{opt}})) + \frac{\beta\xi^{\text{opt}}}{1 - \beta(1-\xi^{\text{opt}})} = 0
                $ reduces to
                \begin{align*}
                    \ln(\xi^{\text{opt}}) + \frac{\xi^{\text{opt}}}{1-(1-\xi^{\text{opt}})} = \ln(\xi^{\text{opt}}) + 1 = 0,
                \end{align*}
                which gives $\xi^{\text{opt}} = e^{-1} = 1/e$. Therefore,
                $T^{\text{opt}} = e\underline{T}$.

                \textit{Case $\beta \ll 1$.} When $\overline{c} \gg \min\{1,\kappa_J\}$, we have $\beta(1-\xi) \ll 1$. Using the Taylor expansion $\ln(1-x) \approx -x - \frac{x^2}{2}$ for small $x$, equation  equation $\ln(1 - \beta(1-\xi^{\text{opt}})) + \frac{\beta\xi^{\text{opt}}}{1 - \beta(1-\xi^{\text{opt}})} = 0$ yields
                \begin{align*}
                    -\beta(1-\xi^{\text{opt}}) + \frac{\beta\xi^{\text{opt}}}{1} \approx 0.
                \end{align*}
                This simplifies to $-\beta + \beta\xi^{\text{opt}} + \beta\xi^{\text{opt}} \approx 0$, giving $\xi^{\text{opt}} \approx 1/2$. Therefore, $T^{\text{opt}} \approx 2\underline{T}$.

                \textit{Case $\beta \in (0,1)$.} For intermediate values of $\beta$, the implicit function theorem guarantees that $\xi^{\text{opt}}(\beta)$ varies continuously from $1/e$ (when $\beta = 1$) to $1/2$ (as $\beta \to 0$). Therefore, $T^{\text{opt}} = \frac{\underline{T}}{\xi^{\text{opt}}} \in [2\underline{T}, e\underline{T}]$.
            \end{proof}}
    }
\begin{remark}\label{rem:semi_to_improved_rates}
    \added{Theorem~\ref{thm:hybrid_stability} yields the \emph{semi-acceleration} bound $\mathcal{O}(1/\tau(t,j)^2)$ along flow intervals, matching the classical $\mathcal{O}(1/t^2)$ rate of Nesterov-type ODEs with damping $3/t$ for smooth convex objectives~\cite{suDifferentialEquationModeling2016} during intervals of flow.}
    \added{In low-curvature regimes ($\kappa_J<1$),} Corollary~\ref{cor:optimalRestart} improves upon gradient-like dynamics $\dot{x}=-\cal G(x)$, which attain $\mathcal{O}(e^{-\kappa_J t})$ convergence rates under Assumption~\ref{assump:regularity_Helmholtz}-(i,ii), and refines~\cite[Lem.~5]{ochoaMomentumBasedNashSetSeeking2024} by removing the factor $\sigma \coloneqq (\ell_J+\ell_{\kappa})/\kappa_J$ from the exponent in $\mathcal{O}\!\big(e^{-(\sqrt{\kappa_J}/\sigma)\,t}\big)$.
    \added{When $\cal G=\nabla J$, existing acceleration results achieve exponential convergence with exponent proportional to $\sqrt{\kappa_J}$~\cite[Thm.~1]{shiUnderstandingAccelerationPhenomenon2022}; our result attains the same qualitative $\sqrt{\kappa_J}$ scaling without the gradient-structure assumption.}

\end{remark}

\begin{figure*}[t]\label{fig:instability}
    \centering
    \includegraphics[width=0.95\linewidth]{"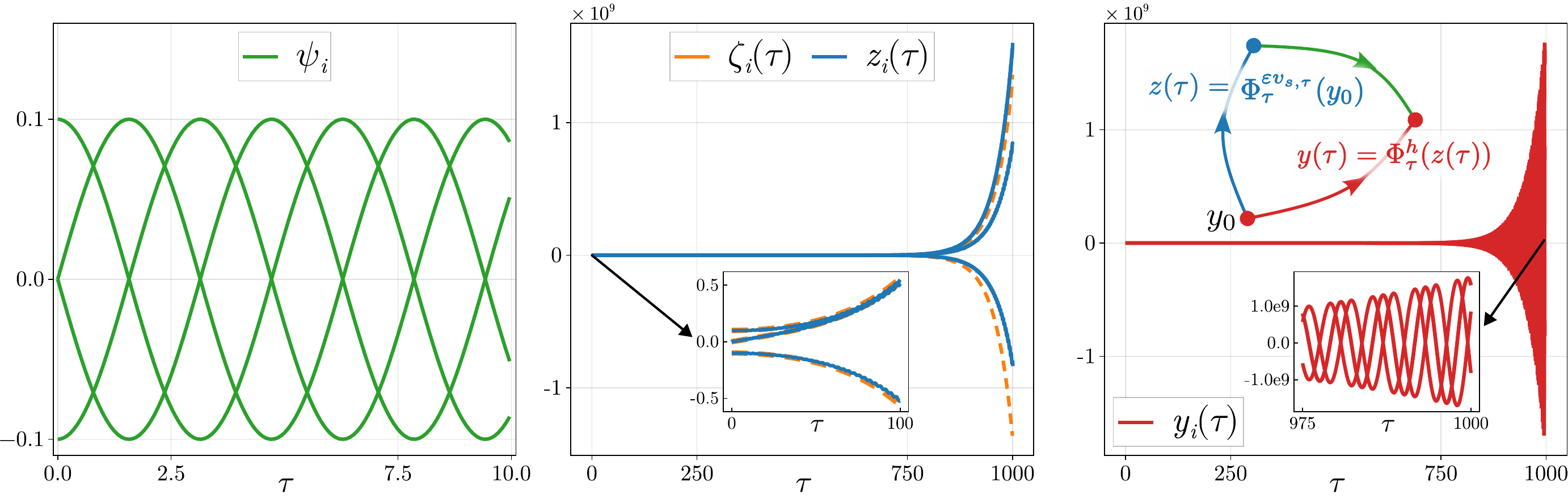"}
    \caption{Left: Components of $\psi$ showing a solution of the periodic system \eqref{eq:driftDynamics}. Middle: Components of $z$ and $\zeta$ showing the solutions of the pulled-back system \eqref{eq:standardLinear} and the averaged system \eqref{proof:average}. Right: Components of $y$ showing a solution of Nesterov's ODE in \eqref{eq:preStandardLinear}.}\vspace{-1.25em}
\end{figure*}
\section{Numerical Simulations}\label{sec:numeric}
\emph{Instability Example:} We consider a non-conservative driving term $\cal G(x)=(Q_s + Q_a)x = \ell_J\left(\hat{Q}_s + \frac{1}{\sqrt{\ell_J}}\hat{Q}_a\right)x$ with
$
    \hat{Q}_s\coloneqq\begin{pmatrix}
        1 & 0 \\
        0 & 1
    \end{pmatrix},~\hat{Q}_a\coloneqq \begin{pmatrix}
        0    & 0.5 \\
        -0.5 & 0
    \end{pmatrix},
$
and $\ell_J=100$, leading to $\varepsilon=0.1$. With these choices, all the conditions of Theorem \ref{maintheorem} are satisfied. We set $\tau(0)=T_0=0.1$, $y(0)=z(0)=(0.1, -0.1, 0, 0)$ for systems \eqref{eq:preStandardLinear}-\eqref{eq:standardLinear}. We also simulate system \eqref{eq:driftDynamics} from $\psi(0)=y(0)$ and the average system \eqref{proof:average} from  $\zeta(0)=z(0)$. The trajectories are shown in Figure 1. The left \added{plot} shows the periodic behavior of the linear system \eqref{eq:driftDynamics}, where the $\psi$-components exhibit bounded oscillations consistent with the periodic flow predicted by Lemma \ref{lemma:periodic}. The middle \added{plot} validates the averaging analysis by showing how the solution $z(\tau)$ of the pulled-back system \eqref{eq:standardLinear} closely follows the solution $\zeta(\tau)$ of the averaged system \eqref{proof:average} on the timescale $\mathcal{O}(1/\varepsilon)$. The right \added{plot} confirms the instability predicted by Theorem \ref{maintheorem}.

\emph{Precluding Instability via Restarting:} We simulate the restarting HDS $\cal H$ using the same vector field $\cal G(x) = (Q_s + Q_a)x$, which yields $\ell_K = 5$, $\ell_J = 100$, and $\kappa_J = 100$. We set $\eta=1/2$, $T_0 = 0.1$ and $T = T^{\text{opt}} = 0.471$, where $T$ satisfies the resetting frequency condition $\underline{T} < T<\overline{T}$ of Theorem \ref{thm:hybrid_stability}. Using $\chi(0,0) = (10^4, -10^4, 10^4, -10^4, 0.1)$, Figure 2 shows the evolution of $|q(t)-x^*|$ for both Nesterov's ODE and the HDS $\cal H$. The red stars indicate the restarting events. The simulations confirm both the UGES property of Theorem \ref{thm:hybrid_stability} and the  $\mathcal{O}(e^{-\sqrt{\kappa_J}t})$ convergence rate of Corollary \ref{cor:optimalRestart}.
\begin{figure}[t]
    \centering
    \includegraphics[width=0.875\linewidth]{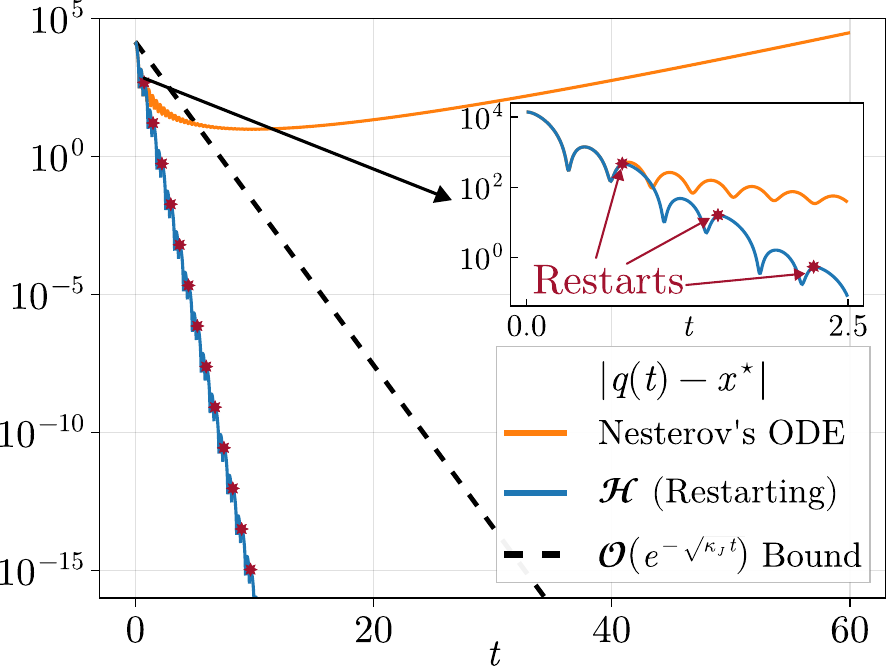}
    \caption{Comparison of trajectories of $|q(t) - x^*|$ under Nesterov's ODE and the restarting HDS $\cal H$.}\vspace{-1.3em}
    \label{fig:restarting}
\end{figure}
\section{Conclusions and Future Directions}\label{sec:conclusions}
In this paper, we prove that for linear vector fields with nonzero skew-symmetric components, Nesterov's ODE can exhibit instability even under strong monotonicity conditions. To \added{resolve} this instability, we design a hybrid dynamical system that \added{achieves robust} stability while \added{inducing} convergence rates of order $\mathcal{O}(e^{-\sqrt{\kappa_J}t})$ through periodic restarting mechanisms. Future work will focus on Lie bracket expansions of the pullback operator to analyze the instability of Nesterov's ODE in the case of nonlinear vector fields.

\renewcommand{\UrlFont}{\scriptsize\ttfamily} 
\renewcommand*{\bibfont}{\footnotesize}  
\DeclareFieldFormat{title}{\footnotesize{#1}}
\printbibliography

\clearpage
\reportOnly{
    \useRomanappendicesfalse
    \renewcommand{\thethmcounter}{\thesection.\arabic{thmcounter}}

    \appendices
    \section{Auxiliary Lemmas}
     {\color{black}
      \begin{lemma}\label{lem:unique_optimum}
          Let $\theta:(0,1) \to \mathbb{R}$ be defined by $\theta(\xi) \coloneqq -\xi\ln(1 - \beta(1-\xi))$ for all $\xi\in(0,1)$ where $\beta\in (0,1]$. Then, $\theta$ has a unique maximizer $\xi^{\text{opt}} \in (0,1)$ satisfying
          \begin{align}\label{eq:optimality_condition:lemma}
              \ln(1 - \beta(1-\xi^{\text{opt}})) + \frac{\beta\xi^{\text{opt}}}{1 - \beta(1-\xi^{\text{opt}})} = 0.
          \end{align}
      \end{lemma}

      \begin{proof}
          Computing the derivative of $\theta$ with respect to $\xi$,
          \begin{align*}
              \frac{d\theta}{d\xi}(\xi)
               & = -\ln(1 - \beta(1-\xi)) - \frac{\beta\xi}{1-\beta(1-\xi)}\eqqcolon -\vartheta(\xi).
          \end{align*}
          Setting $\frac{d\theta}{d\xi}(\xi) = 0$ gives equation \eqref{eq:optimality_condition:lemma}.

          To verify uniqueness, we analyze the behavior of $\vartheta$ on $(0,1)$. Note that $\vartheta(\xi) = 0$ if and only if $\frac{d\theta}{d\xi}(\xi) = 0$. Computing the derivative of $\vartheta$ with respect to $\xi$ yields
          \begin{align*}
              \frac{d\vartheta}{d\xi}(\xi)
               & = \frac{\beta}{1-\beta(1-\xi)}                                          \\
               & \quad+ \frac{\beta(1-\beta(1-\xi)) + \beta^2\xi}{(1-\beta(1-\xi))^2}    \\
               & = \frac{\beta(1-\beta(1-\xi)) + \beta + \beta^2\xi}{(1-\beta(1-\xi))^2} \\
               & = \frac{\beta(2 - \beta + 2\beta\xi)}{(1-\beta(1-\xi))^2} > 0
          \end{align*}
          for all $\xi ,\beta\in (0,1)$. Thus $\vartheta$ is strictly increasing on $(0,1)$.

          Next, as $\xi \to 0^+$, we have that
          \begin{align*}
              \vartheta(0^+) = \ln(1-\beta) + 0 = \ln(1-\beta) < 0
          \end{align*}
          since $\beta \in (0,1]$. As $\xi \to 1^-$:
          \begin{align*}
              \lim_{\xi \to 1^-} \vartheta(\xi)
              = \ln(1) + \lim_{\xi \to 1^-}\frac{\beta\xi}{\beta\xi} = \beta > 0.
          \end{align*}
          By the intermediate value theorem and strict monotonicity of $\vartheta$, there exists a unique $\xi^{\text{opt}} \in (0,1)$ such that $\vartheta(\xi^{\text{opt}}) = 0$.

          Finally, we verify that $\xi^{\text{opt}}$ is a maximum. For any $\xi \in (0,1)$, we have $1 - \beta(1-\xi) \in (0,1)$ since $\beta \in (0,1]$, which implies $\ln(1 - \beta(1-\xi)) < 0$. Therefore,
          \begin{align*}
              \theta(\xi) = -\xi\ln(1 - \beta(1-\xi)) > 0
              \quad \text{for all } \xi \in (0,1).
          \end{align*}
          We have $\theta(0^+) = \lim_{\xi \to 0^+} -\xi\ln(1-\beta) = 0$ and $\theta(1^-) = \lim_{\xi \to 1^-} -\xi\ln(1) = 0$. Since $\theta$ is continuous on $(0,1)$, positive on $(0,1)$, and vanishes at the boundaries, it must attain its maximum at an interior critical point. Since $\xi^{\text{opt}}$ is the unique critical point in $(0,1)$, it is the global maximizer.
      \end{proof}}
}

\end{document}